\documentclass[12pt]{amsart}
\usepackage{amsmath,amssymb,amsbsy,amsfonts,latexsym,amsopn,amstext,cite,
                                               amsxtra,euscript,amscd,bm,mathabx,mathrsfs, abraces}
\usepackage{url}
\usepackage[colorlinks,linkcolor=blue,anchorcolor=blue,citecolor=blue,backref=page]{hyperref}
\usepackage{color}
\usepackage{graphics,epsfig}
\usepackage{graphicx}
\usepackage{float} 
\usepackage[english]{babel}
\usepackage{mathtools}
\usepackage{todonotes}
\usepackage{url}
\usepackage[colorlinks,linkcolor=blue,anchorcolor=blue,citecolor=blue,backref=page]{hyperref}

\usepackage[norefs,nocites]{refcheck}

\hypersetup{breaklinks=true}

\usepackage[norefs,nocites]{refcheck}
\usepackage[english]{babel}
\begin{document}

\newtheorem{thm}{Theorem}
\newtheorem{lem}[thm]{Lemma}
\newtheorem{claim}[thm]{Claim}
\newtheorem{cor}[thm]{Corollary}
\newtheorem{prop}[thm]{Proposition} 
\newtheorem{definition}[thm]{Definition}
\newtheorem{question}[thm]{Open Question}
\newtheorem{conj}[thm]{Conjecture}
\newtheorem{rem}[thm]{Remark}
\newtheorem{prob}{Problem}

\newtheorem{ass}[thm]{Assumption}

\newtheorem{lemma}[thm]{Lemma}

\newcommand{\GL}{\operatorname{GL}}
\newcommand{\SL}{\operatorname{SL}}
\newcommand{\lcm}{\operatorname{lcm}}
\newcommand{\ord}{\operatorname{ord}}
\newcommand{\Tr}{\operatorname{Tr}}
\newcommand{\Span}{\operatorname{Span}}

\numberwithin{equation}{section}
\numberwithin{thm}{section}
\numberwithin{table}{section}

\def\vol {{\mathrm{vol\,}}}
\def\squareforqed{\hbox{\rlap{$\sqcap$}$\sqcup$}}
\def\qed{\ifmmode\squareforqed\else{\unskip\nobreak\hfil
\penalty50\hskip1em\null\nobreak\hfil\squareforqed
\parfillskip=0pt\finalhyphendemerits=0\endgraf}\fi}

\def \balpha{\bm{\alpha}}
\def \bbeta{\bm{\beta}}
\def \bgamma{\bm{\gamma}}
\def \blambda{\bm{\lambda}}
\def \bchi{\bm{\chi}}
\def \bphi{\bm{\varphi}}
\def \bpsi{\bm{\psi}}
\def \bomega{\bm{\omega}}
\def \btheta{\bm{\vartheta}}
\def \bmu{\bm{\mu}}
\def \bnu{\bm{\nu}}

\newcommand{\bfxi}{{\boldsymbol{\xi}}}
\newcommand{\bfrho}{{\boldsymbol{\rho}}}

\def\cA{{\mathcal A}}
\def\cB{{\mathcal B}}
\def\cC{{\mathcal C}}
\def\cD{{\mathcal D}}
\def\cE{{\mathcal E}}
\def\cF{{\mathcal F}}
\def\cG{{\mathcal G}}
\def\cH{{\mathcal H}}
\def\cI{{\mathcal I}}
\def\cJ{{\mathcal J}}
\def\cK{{\mathcal K}}
\def\cL{{\mathcal L}}
\def\cM{{\mathcal M}}
\def\cN{{\mathcal N}}
\def\cO{{\mathcal O}}
\def\cP{{\mathcal P}}
\def\cQ{{\mathcal Q}}
\def\cR{{\mathcal R}}
\def\cS{{\mathcal S}}
\def\cT{{\mathcal T}}
\def\cU{{\mathcal U}}
\def\cV{{\mathcal V}}
\def\cW{{\mathcal W}}
\def\cX{{\mathcal X}}
\def\cY{{\mathcal Y}}
\def\cZ{{\mathcal Z}}
\def\Ker{{\mathrm{Ker}}}

\def\sA{{\mathscr A}}

\def\sssum{\mathop{\sum\!\sum\!\sum}}
\def\ssum{\mathop{\sum\ldots \sum}}
\def\dsum{\mathop{\quad \sum \qquad \sum}}

\newcommand{\abs}[1]{\left| #1 \right|}
\newcommand{\norm}[1]{\left\| #1 \right\|}

\def\Xm{\cX_m}

\def \A {{\mathbb A}}
\def \B {{\mathbb A}}
\def \C {{\mathbb C}}
\def \F {{\mathbb F}}
\def \G {{\mathbb G}}
\def \L {{\mathbb L}}
\def \K {{\mathbb K}}
\def \Q {{\mathbb Q}}
\def \R {{\mathbb R}}
\def \Z {{\mathbb Z}}
\def \P {{\mathbb P}}

\def \vx {\mathbf x}
\def \vy {\mathbf y}
\def \vz {\mathbf z}
\def \vH {\mathbf H}
\def \vU {\mathbf U}
\def \vX {\mathbf X}
\def \vY {\mathbf Y}
\def \vZ {\mathbf Z}

\def \fA{\mathfrak A}
\def \fC{\mathfrak C}
\def \fL{\mathfrak L}
\def \fR{\mathfrak R}
\def \fS{\mathfrak S}

\def \fUg{{\mathfrak U}_{\mathrm{good}}}
\def \fUm{{\mathfrak U}_{\mathrm{med}}}
\def \fV{{\mathfrak V}}
\def \fG{\mathfrak G}
\def \f{\mathfrak G}

\def\e{{\mathbf{\,e}}}
\def\ep{{\mathbf{\,e}}_p}
\def\eq{{\mathbf{\,e}}_q}

 \def\\{\cr}
\def\({\left(}
\def\){\right)}
\def\fl#1{\left\lfloor#1\right\rfloor}
\def\rf#1{\left\lceil#1\right\rceil}

\def\Im{{\mathrm{Im}}}

\def \oF {\overline \F}

\def\vm{\vec{m}}

\newcommand{\pfrac}[2]{{\left(\frac{#1}{#2}\right)}}

\def \Prob{{\mathrm {}}}
\def\e{\mathbf{e}}
\def\ep{{\mathbf{\,e}}_p}
\def\epp{{\mathbf{\,e}}_{p^2}}
\def\edd{{\mathbf{\,e}}_{d^2}}
\def\em{{\mathbf{\,e}}_m}

\def\Res{\mathrm{Res}}
\def\Orb{\mathrm{Orb}}

\newcommand{\rank}{\operatorname{rk}}

\def\vec#1{\mathbf{#1}}
\def \va{\vec{a}}
\def \vb{\vec{b}}
\def \vc{\vec{c}}
\def \vs{\vec{s}}
\def \vu{\vec{u}}
\def \vv{\vec{v}}
\def \vw{\vec{w}}
\def\vlam{\vec{\lambda}}
\def\flp#1{{\left\langle#1\right\rangle}_p}

\def\sE {\mathscr {E}}
\def\sM {\mathscr {M}}

\def\mand{\qquad\mbox{and}\qquad}

\title[Matrices with square-free determinants]
{Counting integer matrices with square-free determinants}

\author[A. Ostafe] {Alina Ostafe}
\address{School of Mathematics and Statistics, University of New South Wales, Sydney NSW 2052, Australia}
\email{alina.ostafe@unsw.edu.au}

\author[I. E. Shparlinski] {Igor E. Shparlinski}
\address{School of Mathematics and Statistics, University of New South Wales, Sydney NSW 2052, Australia}
\email{igor.shparlinski@unsw.edu.au}

\begin{abstract}  We consider the set  $\cM_n\(\Z; H\)$ of $n\times n$-matrices with 
integer elements of size at most $H$ and obtain an asymptotic formula for the number of
matrices from $\cM_n\(\Z; H\)$ with square-free determinants.
We also use our approach with some further enhancements, to obtain an asymptotic formula
for the sums of the Euler function with determinants of matrices from  $\cM_n\(\Z; H\)$.
 \end{abstract}



\maketitle

\tableofcontents

%---------------------------------------------------------------------
\section{Introduction}
\subsection{Motivation results}
For a positive integer $n$, we use $\cM_n\(\Z\)$  to denote the
set of all  $n\times n$ matrices with integer elements. 
Furthermore, for  an  integer $H\ge 1$ we use $\cM_n\(\Z; H\)$  to denote the
set  of matrices
\[
A = \(a_{ij}\)_{i,j=1}^n \in \cM_n\(\Z\)
\]  
with integer entries of size $|a_{ij}| \le H$.  In particular, $\cM_n\(\Z; H\)$ is 
of cardinality $\# \cM_n\(\Z; H\) = \(2H +1\)^{n^2}$. 

Recently there has been quite active interest in investigating arithmetic properties of matrix determinants and minors, see~\cite{KoWo,Kov}.  For example,  Kotsovolis 
and Woo~\cite{KoWo} obtained an asymptotic formula for the number of matrices (counted 
by an slightly different norm) with prime determinants. The methods of homogeneous dynamics, underlying the approach of~\cite{KoWo}, do not allow to get strong bounds on 
the error terms of this asymptotic formula.  Motivated by this work, and also by work of 
Kovaleva~\cite{Kov} on square-free minors, we obtain an asymptotic formula for 
the number $S_n(H)$ 
of matrices $A \in \cM_n\(\Z; H\)$ with a 
square-free determinant, that is, matrices for which $\det A$ is not divisible by 
a square of any prime. While investigating square-free values is typically easier than 
investigating prime values, our goal is to get strong bounds on the error term of our asymptotic formula. 

%%Motivated by recent investigation of some arithmetic properties of matrix determinants and 
%%minors, see~\cite{KoWo,Kov}, here we obtain an asymptotic formula for 
%%the number $S_n(H)$ 
%%of matrices $A \in \cM_n\(\Z; H\)$ with a 
%%square-free determinant, that is, matrices for which $\det A$ is not divisible by 
%%a square of any prime. 

We also remark that treating the determinant as a generic polynomial over $\Z$ in $n^2$ variables and of 
degree $n$, is not going to bring any results as at the present time all results on square-free values of 
polynomials of high degree are conditional on the celebrated $abc$-conjecture, see~\cite{Poon,Reus}, 
or require an exponentially large (compared to the degree) number of variables~\cite{DeSo}.

It is natural to ask whether our results extend to arbitrary multilinear polynomials. Unfortunately the answer is negative as our arguments rely on the special structure 
of determinants, for example, on its invariance with respect to 
standard row operations used in the proof of Lemma~\ref{lem: LinForm=Det=0},
and homogeneity with respect to the variables from the same row,
used in the proof of Lemma~\ref{lem: SpL}. Furthermore, there is no analogue of  Lemma~\ref{lem: Det} for arbitrary multivariate polynomials.

\subsection{Main results}
Our asymptotic formula gives a power saving in the error term, exceeding $1/2$, 
which is the limit of 
approaches based  on the use of the Weil bound. 

In fact, using the standard inclusion-exclusion principle together with  Lemmas~\ref{lem: Det = 0 Mod d^2}  
and~\ref{lem: Det} below one can already get 
a power saving in the error term. However we go beyond this and get a stronger 
bound. 

\begin{thm}
\label{thm:SF-Det}
We have 
\[
S_n(H) = 2^{n^2} \fS_n H^{n^2} + O\(H^{n^2 - \gamma_n}\),
\]
where 
\[ 
\fS_n =  \prod_{p~ \text{prime}} \, \prod_{j = 2}^{n+1} \(1-p^{-j}\) 
\]
and 
\[
\gamma_n  = \frac{1}{2} +
\frac{n-1}{2\(n^3 +3n^2 -n+1\)}.
\]
\end{thm}  

 Note that 
\[
\fS_n =  \prod_{j=2}^{n+1} \zeta(j)^{-1}, 
\] 
where $\zeta(s)$ is the Riemann zeta-function.

Our approach is based on new bounds on exponential sums along determinant hypersurfaces, which 
we believe could be of independent interest. 

\begin{rem}\rm{
Note that the bounds based on  algebraic geometry, 
such as in~\cite{Fouv, Katz}, 
do not apply to rational exponential  sums modulo $p^2$, where $p$ is prime, which is our principal case.
Moreover, even in the case of   exponential sums  modulo $p$, our 
bounds are  stronger that those which can be derived 
from, for example, results of Katz~\cite{Katz-SingSum}. This is due to the high dimension of the singularity locus 
of the determinant variety over $\F_p$ (which is  the set of matrices of $n \times n$ matrices over $\F_p$
of rank at most $n-2$, and thus a variety of dimension $n^2-4$, see~\cite{Ab}), affecting very adversely} the strength of the algebraic geometry approach. 
\end{rem}   

\begin{rem}\rm{
It is also interesting to note  that Browning,  Sawin and Wang~\cite{BSW}
have also used bounds on  exponential sums with matrices over $\F_p$ to establish 
some new results on matrices with integer entries. However the exponential sums which 
arise in~\cite{BSW} and the techniques employed are very different from ours.} 
\end{rem} 

We emphasise that the main feature of Theorem~\ref{thm:SF-Det} is that $\gamma_n > 1/2$
since with 
\begin{equation}
\label{eq:weak gamma}
\gamma_n = n^2/(2n^2+2) < 1/2
\end{equation}
 it can be derived via the aforementioned elementary argument, 
avoiding the use of exponential sums.

Next, to illustrate other  possible applications of our results, in particular,  to obtaining asymptotic formulas for certain sums with multiplicative functions with determinants, we consider the following sum 
\[
\varPhi_n(H) = \sum_{\substack{A \in \cM_n\(\Z; H\)\\\det A \ne 0}}\frac{\varphi\(|\det A|\)}{|\det A|},
\] 
where  $\varphi(k)$ is the Euler function. Our results certainly apply to sums of many other multiplicative functions.

In fact in this case we are able to introduce further  enhancements to the argument underlying the proof 
of Theorem~\ref{thm:SF-Det} and obtain the following result.

\begin{thm}
\label{thm:Mob-Det}
We have 
\[
\varPhi_n(H) =  2^{n^2} \sigma_n H^{n^2} + O\(H^{n^2 -\vartheta_n + o(1)}\),
\]
where 
\[ 
\sigma_n =   \prod_{p~ \text{prime}}\,  \(1- \frac{1}{p}\) \(1+  \frac{1}{p} \prod_{j = 2}^{n} \(1-p^{-j}\) \)
\]
and 
\[
\vartheta_n  = 1 -  \frac{1}{ n^3 +1}.
\]
\end{thm}

\begin{rem}\rm{Our proof of Theorem~\ref{thm:Mob-Det} takes advantage of stronger bounds of exponential sums with square-free denominators with so-called
monomial linear forms, see Lemma~\ref{lem: Sd bound} below. 
We have not been able to establish a  bound of similar strength in the case of denominators, 
which are squares of  square-free integers, see, however, Lemma~\ref{lem: Sd2 bound}.}
%We have not been able to establish an analogue of  this bound in the 
%case of denominators which are squares of  square-free integers.}
\end{rem} 

We also recall that Hafner, McCurley and Sarnak~\cite{HSM}, as an application of their result on 
counting coprime values of polynomials, have studied the probability of coprimality of determinants 
of two random matrices. 
Namely,  we consider the following cardinality: 
\[
R_n(H) = \# \{(A,B) \in \cM_n(H):~ \gcd\(\det A, \det B\) = 1\}.
\]
It is shown in~\cite{HSM} that 
\begin{equation}
\label{eq:HMS}
R_n(H) =2^{n^2}  \(\rho_n + o(1)\)H^{2n^2}
\end{equation}
as $H \to \infty$, where 
\[
\rho_n =  \prod_{p~ \text{prime}}\,  \(1- \(1- \prod_{j = 1}^{n} \(1-p^{-j}\) \)^2\).
\]
We now establish a more precise form of~\eqref{eq:HMS} with a power saving in the error term. 

\begin{thm}
\label{thm:GCD-Det}
We have 
\[
R_n(H) =  2^{2n^2} \rho_n H^{2n^2} + O\(H^{2n^2 -\kappa_n + o(1)}\),
\] 
where  
\[
\kappa_n  = 1 +
\frac{n-1}{n^3  -n+1}.
\]
\end{thm}

\subsection{Notation and conventions}
We recall that  the notations $U = O(V)$, $U \ll V$ and $ V\gg U$  
are equivalent to $|U|\leqslant c V$ for some positive constant $c$, 
which throughout this work, may depend only on $n$. 

We also write $U = V^{o(1)}$ if for any fixed $\varepsilon$ we have 
$V^{-\varepsilon} \le |U |\le V^{\varepsilon}$ provided that $V$ is large 
enough. 

For a finite set $\cS$ we use $\# \cS$ to denote its cardinality. 

For a positive  integer $m$, we use $\Z_m$ to denote the residue ring modulo $m$, which we assume to be represented by the 
set $\{0, \ldots, m-1\}$, and use $\Z_m^*$ for its group of units.
%Whenever clear from the context, for singular matrices $\vX\in \Z_m^{n\times n}$ we write $\det \vX=0$ rather than using congruences modulo $m$. 

The letter $p$ always denotes a prime number.

We also freely alternate between the language of residue rings $\Z_p$ 
modulo $p$ and finite fields $\F_p$ of $p$ elements. Similarly, we also alternate, 
where convenient, between the language of residue rings $\Z_m$ and congruences modulo $m$.

For an integer $k\ne 0$ we denote by $\tau(k)$ the number of positive integer  divisors of $k$, 
for which we very often use the well-known bound
\begin{equation}
\label{eq:tau}
\tau(k) = |k|^{o(1)} 
\end{equation}
as $|k| \to \infty$, see~\cite[Equation~(1.81)]{IwKow}.

We also use the convention that if $A$ is an $n \times n$ matrix defined over $\Z_m$ or $\F_p$, then by $\det A=0$ we mean $\det A\equiv 0 \bmod m$ or $\det A\equiv 0\bmod p$, respectively, which should be clear from the context. When $A$ has entries in an interval $[-H,H]$, then any condition on the determinant modulo a positive integer $m$ or prime $p$ will be clearly indicated.

\section{Preliminaries}

\subsection{Matrices with  determinants with a divisibility condition}
\label{sec:div det} 
Let $N_n(m)$ denote the number of $n \times n$ matrices $A$  over 
$\Z_m$ for which $\det A=0$.

The following  exact formula is a very special case of a much more general 
result of Brent and McKay~\cite[Corollary~2.2]{BreMcK} combined with the 
classical Chinese Remainder Theorem. 

\begin{lemma}
\label{lem: Det = 0 Mod d^2}
For any square-free integer $d\ge 1$,  we have
\[
N_n(d) = d^{n^2} \prod_{p\mid d} \(1 - \prod_{j = 1}^{n} \(1-p^{-j}\)\)
\]
and
\[
N_n(d^2) = d^{2n^2} \prod_{p\mid d} \(1 - \prod_{j = 2}^{n+1} \(1-p^{-j}\)\).
\]
\end{lemma}

\subsection{Matrices with fixed determinants}
 
We need a bound on the number of matrices $A \in \cM_n\(\Z;H\)$ with prescribed value 
of the determinant $\det A =d$. We recall that  Duke,  Rudnick and  Sarnak~\cite{DRS},
if $d \ne 0$, and  Katznelson~\cite{Katz}, when $d=0$, have obtained asymptotic formulas 
(with the main terms of orders $H^{n^2 - n}$ and $H^{n^2 - n}\log H$, respectively) 
for the number of  such matrices when $d$ is fixed. However this is too restrictive for our purpose  because we need  a uniform with respect to $d$ results, while in~\cite{DRS} the dependency on $d$ is not specified. Hence we use 
a uniform with respect to $d$ upper bound which is a special case 
of~\cite[Theorem~4]{Shp}.

\begin{lemma}
\label{lem: Det}
Uniformly over $a \in \Z$, there are at most  $O\(H^{n^2 - n}\log H\)$ matrices    $A \in \cM_n\(\Z;H\)$ with 
 $\det A =a$.
\end{lemma}

In principle, as we have mentioned, a combination of Lemmas~\ref{lem: Det = 0 Mod d^2}
and~\ref{lem: Det} is already 
enough to prove a version of Theorem~\ref{thm:SF-Det}  with $\gamma_n$ given by~\eqref{eq:weak gamma}.
However,  to get the desired bound on the error term we also need to count 
matrices  $A \in \Z_{d^2}^{n \times n}$  
with entries in incomplete intervals 
in $\Z_{d^2}$ such that $\det A =0$.
  This is done in Section~\ref{sec: ExpDet} via bounds of exponential sums. 

 \section{Exponential sums along determinant congruences}  
 \label{sec: ExpDet}
 
 \subsection{Set-up}
 
 Given a linear form 
\begin{equation}
\label{eq: LinForm}
 L(\vX)  = \sum_{i,j = 1}^n a_{ij} x_{ij}  \in \Z[\vX]
\end{equation}
 in $n^2$ variables $\vX = \(x_{ij}\)_{i,j=1}^n$ and an integer $m$, we define the 
 exponential sums 
 \[
 S_m(L) = \sum_{\det \vX \equiv 0 \bmod m} \em\( L(\vX)\)
 \]
where $\em(u) = \exp(2\pi i u/m)$, 
over  solutions to the congruence 
$\det \vX =0$  in $\vX \in \Z_m^{n \times n}$. 

We are only interested in the sums $ S_m(L)$ when  the modulus $m = d$ or $m = d^2$
for a square-free integer $d \ge 1$. Thus we start with reduction to 
sums modulo $p$ or $p^2$. Namely, by the Chinese Remainder Theorem, we have
the following decomposition, see~\cite[Equation~(12.21)]{IwKow}.

\begin{lemma}
\label{lem: Sd2 Sp2}
For any square-free integer $d\ge 1$ and a linear form $ L(\vX)  \in \Z[\vX]$, 
we have
\[
S_{d}(L)=   \prod_{p\mid d} S_{p}(b_{d,p}L) \mand S_{d^2}(L)=   \prod_{p\mid d} S_{p^2}(b_{d,p}L), 
\]
for some integers $b_{d,p}$ with $\gcd(b_{d,p},p) = 1$. 
\end{lemma} 

Hence we now concentrate on the sums $S_{p}(L)$ and  $S_{p^2}(L)$. 

 \subsection{Linear sections of the variety of singular matrices}

 Let $K$ be an arbitrary field and let $\cV_n$ be the variety of 
 $n \times n$ singular matrices over $K$.
 
\begin{lemma}
\label{lem: Lin Sec}
The variety $\cV_n$ is not contained in any hyperplane. 
\end{lemma} 

\begin{proof} 
 Assume that $\cV_n$ is contained in a hyperplane $\cH$, 
which is given by the equation 
\[
\sum_{i,j = 1}^n a_{ij} x_{ij} = a_0
\]
with at least one non-zero coefficient, say $a_{k,m} \ne 0$. Considering matrices 
with zero-entries everywhere, except at the $(k,m)$ position, which are clearly singular, we
see that $a_{k,m}  x_{k,m}  = a_0$ for all $x_{k,m}  \in K$, which is clearly impossible
for  $a_{k,m} \ne 0$. 
\end{proof}

We  write  $n \times n$  matrices over $\F_p$ as  $\(\vz \mid\vZ\)$ 
for a vector  $\vz \in \F_p^{n}$ and  an $n \times (n-1)$  matrix
 $ \vZ \in \F_p^{n \times (n-1)}$. We also recall that $N_n(p)$ denotes the number of $n \times n$ matrices $A$  over 
$\F_p$ for which $\det A =0$. 
  
\begin{lemma}
\label{lem: LinForm=Det=0}  Let $\ell(\vz)$ be a nontrivial linear form over $\F_p$. Then the system 
of equations 
\[
\det \(\vz \mid\vZ\)   = \ell(\vz) = 0, \qquad  \(\vz, \vZ\) \in 
 \F_p^{n} \times  \F_p^{n \times (n-1)},
\]
has $p^{n(n-1)} + (p^{n-1}-1)p^{n-1}N_{n-1} (p)$ solutions.
\end{lemma} 

\begin{proof} We  fix a vector $\vz \in \F_p^{n}$  with  $\ell(\vz)=0$. 
Clearly if $\vz = \mathbf{0}$ is the zero-vector, then for any $\vZ \in \F_p^{n \times (n-1)}$
we have $\det \( \mathbf{0}\mid\vZ\)  =0$. Hence there are $p^{n(n-1)}$ such solutions.

Now, let $\vz \ne \mathbf{0}$. Without loss of generality, we can assume 
that $\vz =(z_1, z_2, \ldots, z_n)$ with $z_1 \ne 0$. 
Using elementary row operations, we can reduce the equation 
$\det \(\vz \mid\vZ\) = 0$ to an equation $\det \(\vz_0 \mid \vU\) = 0$ 
for some matrix $\vU$ and with $\vz_0  =(z_1,0, \ldots, 0)$. 
It is now obvious that $\det \(\vz_0 \mid \vU\) = 0$ holds for 
$p^{n-1}N_{n-1} (p)$ matrices $\vU  \in \F_p^{n \times (n-1)}$
with an arbitrary top row and an arbitrary singular $(n-1) \times (n-1)$
matrix over $\F_p$ on the other positions. Observe that  row operations are invertible, 
and so, for a given $\vz$,  each matrix $\vU  \in \F_p^{n \times (n-1)}$ corresponds to a unique matrix $\vZ \in \F_p^{n \times (n-1)}$.
Since there are $p^{n-1}-1$ choices for $\vz \ne \mathbf{0}$ with $\ell(\vz)=0$, there are 
$ (p^{n-1}-1)p^{n-1}N_{n-1} (p)$ such solutions.
\end{proof}  

 \subsection{Bounding  $S_{p}(L)$}
 \label{sec:SpL} 
 
Combining Lemma~\ref{lem: Lin Sec} with  bounds of exponential sums along 
algebraic varieties, see, for example,~\cite[Proposition~1.2]{Fouv}, one can immediately derive
\begin{equation}
\label{eq:SpL  Weil}
S_{p}(L) \ll p^{n^2-3/2}
\end{equation}
for any  nonvanishing linear form $ L(\vX)  \in \F_p[\vX]$.

We prove now stronger results, and we start with the case when $L(\vX)$ depends only on the first column of $\vX$.

 \begin{lemma}
\label{lem: SpL}
For any nonvanishing linear form $\ell(\vz)  \in  \F_p[\vz]$, with $L(\vX) = \ell(\vz)$, 
where $\vX = \(\vz \mid\vZ\)$, 
we have
\[
S_{p}(L) \ll p^{n^2-n}. 
\]
\end{lemma}

\begin{proof} Using that for any  $\lambda \in \F_p^*$, the map $\vz \mapsto \lambda \vz$ 
permutes $ \F_p^{n}$, we can write 
\begin{align*}
S_{p}(L)&=  
 \sum_{\vZ \in \F_p^{n \times (n-1)} } \,   
 \sum_{\substack{\vz \in \F_p^{n}\\  \det \( \vz \mid\vZ\)  =0  } }
 \ep\(\ell(\vz )\) \\
 &  = \frac{1}{p-1} \sum_{\lambda \in \F_p^*} 
\sum_{\vZ \in \F_p^{n \times (n-1)} } \,   
\sum_{\substack{\vz \in \F_p^{n}\\  \det \( \lambda \vz \mid\vZ\)  =0   } }
\ep\(\ell( \lambda\vz )\) \\
&  = \frac{1}{p-1} \sum_{\vZ \in \F_p^{n \times (n-1)} } \,   
 \sum_{\substack{\vz \in \F_p^{n}\\  \det \(\vz \mid\vZ\)   =0  } }
 \sum_{\lambda \in \F_p^*} \ep\( \lambda\ell(\vz )\). 
 \end{align*}
 Hence
 \begin{equation}
\label{eq:S Sigma12}
S_{p}(L)  = \frac{1}{p-1} \(p \Sigma_1 -\Sigma_2\),
\end{equation} 
where 
\begin{align*}
&\Sigma_1 =   \sum_{\vZ \in \F_p^{n \times (n-1)} }  \sum_{\substack{\vz \in \F_p^{n}\\  \det \(\vz \mid\vZ\)   =0  \\ \ell(\vz)  =0 } } 1,\\
&\Sigma_2=  \sum_{\vZ \in \F_p^{n \times (n-1)} } \sum_{\substack{\vz \in \F_p^{n}\\  \det \(\vz \mid\vZ\)  =0   } }1.
\end{align*}

By Lemma~\ref{lem: LinForm=Det=0}, we have
\begin{equation}
\label{eq:Sigma1}
\Sigma_1 = p^{n(n-1)} + \(p^{n-1}-1\)p^{n-1}N_{n-1} (p), 
\end{equation} 
and obviously 
\begin{equation}
\label{eq:Sigma2}\Sigma_2 = N_{n} (p).
\end{equation}
We now recall Lemma~\ref{lem: Det = 0 Mod d^2}, and see that
\begin{align*}
N_{n} (p)/p^{n^2} & =  1 - \prod_{j = 1}^{n} \(1-p^{-j}\) \\
&=  1 -  \(1-N_{n-1} (p)/p^{(n-1)^2}\) \(1-p^{-n}\) \\
&=   \(1-p^{-n}\)  N_{n-1} (p)/p^{(n-1)^2} + p^{-n}.
\end{align*}
Hence 
\[
N_{n} (p) =   \(1-p^{-n}\)  p^{2n-1}  N_{n-1} (p) + p^{n^2-n}. 
\]
Recalling~\eqref{eq:Sigma1} and~\eqref{eq:Sigma2}, we see that 
\begin{align*}
p \Sigma_1 -\Sigma_2 & = p^{n(n-1)+1} + \(p^{n-1}-1\)p^{n}N_{n-1} (p)\\
& \qquad \qquad \quad -  \(1-p^{-n}\) p^{2n-1} N_{n-1} (p)  - p^{n^2-n}\\
& =  (p-1) p^{n(n-1)}  - (p-1) p^{n-1}N_{n-1} (p)  . 
\end{align*}
Since we see from Lemma~\ref{lem: Det = 0 Mod d^2}, that $N_{n-1} (p)   \ll p^{n^2-2n}$
we now conclude that 
\[
p \Sigma_1 -\Sigma_2 =  p^{n(n-1)+1}  + O\( p^{n^2-n}\), 
\]
which after substitution in~\eqref{eq:S Sigma12} concludes the proof. 
\end{proof}

We now extend the bound of Lemma~\ref{lem: SpL}, with a weaker saving but still improving~\eqref{eq:SpL  Weil}, to any nontrivial (that is, not vanishing identically) form $L(\vX)$ over $\F_p$.

\begin{lemma}
\label{lem: SpL1}
For any nontrivial linear form $ L(\vX)  \in \F_p[\vX]$, 
we have 
\[
S_{p}(L) \ll p^{n^2-(n+1)/2}. 
\]
\end{lemma} 

\begin{proof}
Since the form $L(\vX)=\sum_{i,j = 1}^n a_{ij} x_{ij} \in\F_p[\vX]$  is nontrivial,  without loss of generality, we may assume $a_{1j}\ne0$ for some $j=1,\ldots,n$. 
We write every matrix $\vX \in  \F_p^{n \times n}$ as
\[
\vX = \(\vx \mid\vY\)  
\]
with a column  $\vx \in  \F_p^{n}$ and a matrix  $\vY \in  \F_p^{n \times (n-1)}$. 
Therefore, $L(\vX)$ can be written in the form
\[
L(\vX)=\ell(\vx) +L^*(\vY),
\]  
for some linear forms $\ell$ and $L^*$ in $n$ and $n(n-1)$ variables, respectively, where, by our assumption, $\ell$ is nontrivial modulo $p$.

Thus 
\begin{align*}
S_{p}(L)&=  
\dsum_{\substack {\vY \in  \F_p^{n \times (n-1)} \ \vx \in  \F_p^{n}\\ 
\det \(\vx \mid\vY\)  =0}}
 \ep\(\ell(\vx) + L^*(\vY)\)\\
 &=\sum_{\vY \in  \F_p^{n \times (n-1)}}\ep\(L^*(\vY)\)\sum_{\substack {\vx \in  \F_p^{n}\\ 
\det \(\vx \mid\vY\)  =0}} \ep\(\ell(\vx)\).
\end{align*}
Taking absolute values, we obtain
\[
|S_{p}(L)|\le \sum_{\vY \in  \F_p^{n \times (n-1)}} \left|\sum_{\substack{\vx \in  \F_p^{n}\\ 
\det \(\vx \mid\vY\)  =0}} \ep\(\ell(\vx)\)\right|.
\]
By the Cauchy inequality, 
\begin{equation}
\label{eq:SpL^2}
\begin{split}
|S_{p}(L)|^2&\le p^{n^2-n} \sum_{\vY \in  \F_p^{n \times (n-1)}} \left|\sum_{\substack{\vx \in  \F_p^{n}\\ 
\det \(\vx \mid\vY\)  =0}} \ep\(\ell(\vx)\)\right|^2\\
&=p^{n^2-n} \sum_{\vY \in  \F_p^{n \times (n-1)}} \sum_{\substack{\vx_1,\vx_2 \in  \F_p^{n}\\ 
\det \(\vx_i \mid\vY\)  =0,\ i=1,2}} \ep\(\ell(\vx_1-\vx_2)\).
\end{split}
\end{equation}
We now make  the change of variables 
\[
\vx = \vx_1 \mand \vz=\vx_1-\vx_2
\] 
and notice that 
\[
\det \(\vx_1 \mid\vY\)  = \det \(\vx_2 \mid\vY\)  =0
\]
is equivalent to 
\[
\det \(\vx \mid\vY\)  =  \det \(\vx_1 \mid\vY\) =0
\]
and 
\begin{align*}
\det \(\vz \mid\vY\) &=  \det \(\vx_1-\vx_2 \mid\vY\)   \\
& = \det \(\vx_1\mid\vY\) -  \det \(\vx_2\mid\vY\) =0.
\end{align*}
Hence, we see that~\eqref{eq:SpL^2} implies 
\[
|S_{p}(L)|^2\le  p^{n^2-n} \sum_{\vY \in  \F_p^{n \times (n-1)}} \sum_{\substack{\vx \in  \F_p^{n}\\ 
\det \(\vx \mid\vY\)  =0}} \, \sum_{\substack{\vz \in  \F_p^{n}\\ 
\det \(\vz \mid\vY\) =0}} \ep\(\ell(\vz)\).
\]

We now split the summation over $\vY \in  \F_p^{n \times (n-1)} $ 
into two parts depending on whether it is of full rank over $\F_p$ or 
of  rank  $\rank_p \vY \le n-2$. 

We observe that if $\rank_p \vY \le n-2$ then for all $\vz \in  \F_p^{n}$ we have 
$\det \( \vz \mid\vY\)   =0$.  Hence, 
\begin{equation}
\label{eq:Vanishing}
  \sum_{\substack{\vz \in  \F_p^{n}\\  \det \( \vz \mid\vY\)  =0   } }
 \ep\(\ell(\vz )\) = 
   \sum_{\vz \in  \F_p^{n} }
 \ep\(\ell(\vz )\) = 0
\end{equation}
for any nontrivial modulo $p$ linear form $\ell(\vz)$. Therefore, we obtain
\[
|S_{p}(L)|^2\le p^{n^2-n} \sum_{\substack{\vY \in  \F_p^{n \times (n-1)}\\ \rank_p \vY=n-1}}\, \sum_{\substack{\vx \in  \F_p^{n}\\ 
\det \(\vx \mid\vY\) =0}} \, \sum_{\substack{\vz \in  \F_p^{n}\\ 
\det \(\vz \mid\vY\)  =0}} \ep\(\ell(\vz)\).
\]
We note that for $\vY \in  \F_p^{n \times (n-1)}$ with $\rank_p \vY = n-1$ we have 
 \[
 \sum_{\substack {  \vx \in  \F_p^{n}\\   
\det \(\vx \mid\vY\) =0}} 1 = p^{n-1}
\]
since it counts the number of solutions to some nontrivial linear congruence. 
Hence, 
\[
|S_{p}(L)|^2\le p^{n^2-1} \sum_{\substack{\vY \in  \F_p^{n \times (n-1)}\\ \rank_p \vY=n-1}}\,  \sum_{\substack{\vz \in  \F_p^{n}\\ 
\det \(\vz \mid\vY\)  =0}} \ep\(\ell(\vz)\),
\]
where, using the identity~\eqref{eq:Vanishing}, we can add back the terms with $\rank_p \vY \le n-2$ 
and rewrite the above bound as
\[
|S_{p}(L)|^2\le p^{n^2-1} \sum_{\vY \in  \F_p^{n \times (n-1)}}\,  \sum_{\substack{\vz \in  \F_p^{n}\\ 
\det \(\vz \mid\vY\)  =0}} \ep\(\ell(\vz)\). 
\]
Thus, together with Lemma~\ref{lem: SpL}, we   obtain
\begin{align*}
|S_{p}(L)|^2&\le p^{2n^2-n-1},
\end{align*}
which concludes the proof.
\end{proof}

 \subsection{Bounding $S_{p^2}(L)$}
 \label{sec:Sp2L}   A variation of some of the arguments from Section~\ref{sec:SpL}
 also allows us to estimate  $S_{p^2}(L)$.
%% The following bound is our main technical tool. 
 
 \begin{lemma}
\label{lem: Sp2L}
For any  linear form $ L(\vX)  \in \Z[\vX]$, which is  nonvanishing  modulo $p$, 
we have
\[
S_{p^2}(L) \ll p^{2n^2-(n+3)/2}. 
\] 
\end{lemma} 

\begin{proof} The argument below has several similarities with that in the proof of Lemma~\ref{lem: SpL1}, 
however we present it in full detail. 

Since in the proof we work with both elements of  $ \Z_{p^2}$ and  of $\Z_p$,  we use the language 
of congruences rather than of equations in the corresponding rings. 

 Without loss of generality, we can assume that $L(\vX)$ 
is nontrivial modulo $p$ with respect to the first column of $\vX$. 

We write every matrix $\vX \in \Z_{p^2}^{n \times n}$ as
\[
\vX = \(\vx \mid\vY\)  
\]
with a column  $\vx \in \Z_{p^2}^{n}$ and a matrix  $\vY \in \Z_{p^2}^{n \times (n-1)}$.

Thus 
\[
S_{p^2}(L)  = 
\dsum_{\substack {\vY \in \Z_{p^2}^{n \times (n-1)} \ \vx \in \Z_{p^2}^{n}\\ 
\det \(\vx \mid\vY\)  \equiv 0 \bmod {p^2}}}
 \epp\(\ell(\vx) + L^*(\vY)\)
\]
for some linear forms $\ell$ and $L^*$ in $n$ and $n(n-1)$ variables, respectively, 
where by our assumption, $\ell$ is nontrivial modulo $p$. 

Furthermore,  we have
\[
S_{p^2}(L)  = \frac{1}{p^n}   \sum_{\vy \in \Z_{p}^{n}} 
\dsum_{\substack {\vY \in \Z_{p^2}^{n \times (n-1)} \ \vx \in \Z_{p^2}^{n}\\ 
\det \(\vx + p\vy  \mid\vY\)  \equiv 0 \bmod {p^2}}}
 \epp\(\ell(\vx+p\vy ) + L^*(\vY)\).
\]

Next, we  observe that for  $\vx \in \Z_{p^2}^{n}$ and  $\vy \in \Z_{p}^{n}$ we have 
\[
\ell \(\vx  + p \vy\) \equiv \ell(\vx)+p \ell(\vy)   \bmod {p^2}
\]
and  
\begin{equation}
\label{eq:det x y}
 \det \( \vx  + p \vy \mid\vY\) 
\equiv \det \( \vx   \mid\vY\) +p  \det \(  \vy \mid\vY\)  \bmod {p^2} .
\end{equation}
Note that $ \det \( \vx  + p \vy \mid\vY\)   \equiv 0   \bmod {p^2}$ 
implies $ \det \( \vx  \mid\vY\)   \equiv 0   \bmod {p}$, which we write as
$ \det \( \vx  \mid\vY\) = - a_{\vx, \vY} p$ for some integer $a_{\vx, \vY}$,  which is
uniquely defined modulo $p$. Therefore, from~\eqref{eq:det x y}, we conclude that 
\[
\det \(  \vy \mid\vY\) \equiv a_{\vx, \vY} \bmod p.
\]

Thus, changing the order of summation, we write
\begin{align*}
S_{p^2}(L)  = \frac{1}{p^n}  
\dsum_{\substack {\vY \in \Z_{p^2}^{n \times (n-1)} \ \vx \in \Z_{p^2}^{n}\\ 
\det \(\vx \mid\vY\)  \equiv 0 \bmod {p}}} & \epp\(\ell(\vx)) + L^*(\vY)\)\\
& \quad \sum_{\substack{\vy \in \Z_{p}^{n}\\  \det \( \vy \mid\vY\)   \equiv a_{\vx, \vY}    \bmod p } }
 \ep\(\ell(\vy )\), 
\end{align*} 
and therefore 
\[
\abs{S_{p^2}(L)}  \le \frac{1}{p^n}  
\dsum_{\substack {\vY \in \Z_{p^2}^{n \times (n-1)} \ \vx \in \Z_{p^2}^{n}\\ 
\det \(\vx \mid\vY\)  \equiv 0 \bmod {p}}}  \abs{\sum_{\substack{\vy \in \Z_{p}^{n}\\  \det \( \vy \mid\vY\)   \equiv a_{\vx, \vY}    \bmod p } }
 \ep\(\ell(\vy )\) }.
\]

Clearly, 
\[
\dsum_{\substack {\vY \in \Z_{p^2}^{n \times (n-1)} \ \vx \in \Z_{p^2}^{n}\\ 
\det \(\vx \mid\vY\)  \equiv 0 \bmod {p}}} 1
\le p^{n^2} \sum_{\substack {\vX \in \Z_{p}^{n \times n} \\ 
\det \(\vX\)  \equiv 0 \bmod {p}}} 1 \ll  p^{2n^2 -1}. 
\]
Therefore, by the Cauchy inequality
\begin{align*}
S_{p^2}(L)^2 & \ll   p^{2n^2 -2n - 1}
\dsum_{\substack {\vY \in \Z_{p^2}^{n \times (n-1)} \ \vx \in \Z_{p^2}^{n}\\ 
\det \(\vx \mid\vY\)  \equiv 0 \bmod {p}}}  \abs{\sum_{\substack{\vy \in \Z_{p}^{n}\\  \det \( \vy \mid\vY\)   \equiv a_{\vx, \vY}    \bmod p } }
 \ep\(\ell(\vy )\) }^2 \\
 & \ll   p^{2n^2 -2n - 1}
\dsum_{\substack {\vY \in \Z_{p^2}^{n \times (n-1)} \ \vx \in \Z_{p^2}^{n}\\ 
\det \(\vx \mid\vY\)  \equiv 0 \bmod {p}}} \,  \sum_{\substack{\vy_1\vy_2 \in \Z_{p}^{n}\\  \det \( \vy_1 \mid\vY\)   \equiv  a_{\vx, \vY}    \bmod p \\
 \det \( \vy_2 \mid\vY\)   \equiv  a_{\vx, \vY}    \bmod p } }
 \ep\(\ell(\vy_1-\vy_2 )\) .
\end{align*}

Writing $\vz = \vy_1-\vy_2$, we see that 
\[ \det \( \vz \mid\vY\)   \equiv    \det \( \vy_1\mid\vY\)  -  \det \( \vy_2\mid\vY\)   \equiv  0   \bmod p.
\] Hence changing the variables we obtain 
\begin{equation}
\label{eq:S^2}
\begin{split}
S_{p^2}(L)^2  
  \ll   p^{2n^2 -2n - 1}
\dsum_{\substack {\vY \in \Z_{p^2}^{n \times (n-1)} \ \vx \in \Z_{p^2}^{n}\\ 
\det \(\vx \mid\vY\)  \equiv 0 \bmod {p}}} \, & \sum_{\substack{\vy \in \Z_{p}^{n}\\  \det \( \vy  \mid\vY\)   \equiv  a_{\vx, \vY}    \bmod p   } }\\
 & \quad \sum_{\substack{\vz \in \Z_{p}^{n}\\  \det \( \vz \mid\vY\)   \equiv  0  \bmod p   } }
 \ep\(\ell(\vz )\).
\end{split}
\end{equation}
We now split the summation over $\vY \in \Z_{p^2}^{n \times (n-1)} $ 
into two parts depending on whether it is of full rank over $\Z_p$ or 
of  rank  $\rank_p \vY \le n-2$. 

We recall that  if $\rank_p \vY \le n-2$ then we have~\eqref{eq:Vanishing}.
Therefore we now derive from~\eqref{eq:S^2} that 
\begin{equation}
\label{eq:S^2 T}
S_{p^2}(L)^2   \ll  p^{2n^2 -2n - 1} T, 
\end{equation}
where 
\[
T = 
\dsum_{\substack {\vY \in \Z_{p^2}^{n \times (n-1)} \ \vx \in \Z_{p^2}^{n}\\
\det \(\vx \mid\vY\)  \equiv 0 \bmod {p}\\   \rank_p \vY = n-1}} \  \sum_{\substack{\vy \in \Z_{p}^{n}\\  \det \( \vy  \mid\vY\)   \equiv  a_{\vx, \vY}    \bmod p   } }\,
  \sum_{\substack{\vz \in \Z_{p}^{n}\\  \det \( \vz \mid\vY\)   \equiv  0  \bmod p   } }
 \ep\(\ell(\vz )\).
 \]
We note that for $\vY \in \Z_{p^2}^{n \times (n-1)}$ with $\rank_p \vY = n-1$ we have 
 \[
 \sum_{\substack {  \vx \in \Z_{p^2}^{n}\\   
\det \(\vx \mid\vY\)  \equiv 0 \bmod {p}}} 1 = p^{2n-1}
\mand \sum_{\substack{\vy \in \Z_{p}^{n}\\  \det \( \vy  \mid\vY\)   \equiv  a_{\vx, \vY}    \bmod p   } } 1= p^{n-1}
\]
(since both count the number of solutions to some nontrivial linear congruences). 
Hence,  we can simplify the formula for $T$ as 
\[
T = p^{3n-2}
\sum_{\substack {\vY \in \Z_{p^2}^{n \times (n-1)}  \\   \rank_p \vY = n-1}} \,   
  \sum_{\substack{\vz \in \Z_{p}^{n}\\  \det \( \vz \mid\vY\)   \equiv  0  \bmod p   } }
 \ep\(\ell(\vz )\).
 \]
 Next, recalling~\eqref{eq:Vanishing}, we bring back to $T$ the contribution from the matrices $\vY \in \Z_{p^2}^{n \times (n-1)}$ 
 with $\rank_p \vY \le  n-2$.   That is, we can drop the condition $ \rank_p \vY = n-1$ 
 and write
\begin{align*}
T& = p^{3n-2}
\sum_{ \vY \in \Z_{p^2}^{n \times (n-1)} } \,   
  \sum_{\substack{\vz \in \Z_{p}^{n}\\  \det \( \vz \mid\vY\)   \equiv  0  \bmod p   } }
 \ep\(\ell(\vz )\)\\
&  =  p^{3n-2} \cdot p^{n(n-1)} \sum_{\vZ \in \Z_{p}^{n \times (n-1)} } \,   
  \sum_{\substack{\vz \in \Z_{p}^{n}\\  \det \( \vz \mid\vZ\)   \equiv  0  \bmod p   } }
 \ep\(\ell(\vz )\). 
\end{align*}
Invoking Lemma~\ref{lem: SpL}, we obtain 
\[
T \ll   p^{3n-2} \cdot p^{n(n-1)} \cdot p^{n^2 - n} = p^{2n^2 +n- 2} 
\]
and substituting this bound in~\eqref{eq:S^2 T}, 
we  derive the desired result. 
\end{proof}

  \subsection{Bounding $S_{d}(L)$ and $S_{d^2}(L)$}

  We first  combine Lemma~\ref{lem: Sd2 Sp2} with the estimates from Sections~\ref{sec:SpL} and~\ref{sec:Sp2L}   to derive a bound on $S_{d}(L)$ for an arbitrary square-free integer $d$.
  
For a linear form $L$ as in~\eqref{eq: LinForm} and an integer $m\ge 1$, we define
$\gcd(L, m)$ as the greatest common divisor of the coefficients of $L$ and $m$.

We start with an upper bound for $S_{d}(L)$.

\begin{lemma}
\label{lem: Sd bound}
For any square-free integer $d\ge 1$ and a linear form $ L(\vX)  \in \Z[\vX]$, 
we have
\[
|S_{d}(L)| \le  d^{n^2+o(1)} 
\begin{cases} 
(D/d)^n& \text{if $L$ is a monomial,}\\
(D/d)^{(n+1)/2} & \text{otherwise,}
\end{cases} 
\]
where $\gcd(L, d)=D$. 
\end{lemma} 

\begin{proof} 
By Lemma~\ref{lem: Sd2 Sp2},  we have 
\begin{equation}
\label{eq:S S1S2}
S_{d}(L)=  S_1 S_2,
\end{equation}
where
\[
S_1 =  \prod_{p\mid d/D} S_{p}(b_{d,p}L) \mand  S_2 =  \prod_{p\mid  D} S_{p}(b_{d,p}L).
\]
Note that if $p\mid  d/D$ then $\gcd(L, p)=1$. 
In this case, if $L$ is a monomial  (or in fact depends only  on one column), then we apply Lemma~\ref{lem: SpL}, 
and the bound on the divisor function~\eqref{eq:tau},
to obtain
\begin{equation}
\label{eq:Bound S1 mon}
|S_1|  \le (d/D)^{n^2-n +o(1)}.
 \end{equation}
If $L$ has at least two nonzero terms, then we apply Lemma~\ref{lem: SpL1} to obtain
\begin{equation}
\label{eq:Bound S1 non-mon}
|S_1|  \le (d/D)^{n^2-(n+1)/2 +o(1)}.
 \end{equation}
 
 Next, for $S_2$, since $p$ divides all coefficients of $L$, the form $L$ vanishes modulo $p$, and thus 
\begin{equation}
\label{eq:Bound S2 trivial}
  |S_2| =  \prod_{p\mid  D}  p^{n^2} = D^{n^2}.
 \end{equation}

  Substituting the estimates~\eqref{eq:Bound S1 mon},~\eqref{eq:Bound S1 non-mon} and~\eqref{eq:Bound S2 trivial} 
  in~\eqref{eq:S S1S2}, we derive  the result. 
 \end{proof}

We proceed, following similar line of proof as above, to give a bound for $S_{d^2}(L)$. 

\begin{lemma}
\label{lem: Sd2 bound}
For any square-free integer $d\ge 1$ and a linear form $ L(\vX)  \in \Z[\vX]$, 
we have
\[
|S_{d^2}(L)| \le  d^{2n^2-(n+3)/2+o(1)} ef^{(n+3)/2}   
\]
for  some positive square-free integers $e$ and $f$ with $\gcd(e,f)=1$, which are defined 
by  $ \gcd(L, d^2)=  e f^2$. 
\end{lemma} 

\begin{proof} 
As in the proof of Lemma~\ref{lem: Sd bound}, applying Lemma~\ref{lem: Sd2 Sp2},  we have 
\begin{equation}
\label{eq:S S1S2S3}
S_{d^2}(L)=  S_1 S_2 S_3,
\end{equation}
where
\[
S_1 =  \prod_{p\mid d/(ef)} S_{p^2}(b_{d,p}L),\quad S_2 =  \prod_{p\mid  e} S_{p^2}(b_{d,p}L),\quad
S_3 =  \prod_{p\mid f} S_{p^2}(b_{d,p}L).
\]
Note that if $p\mid  d/(ef)$ then $\gcd(L, p)=1$. 
Hence, 
by Lemma~\ref{lem: Sp2L} and the bound on the divisor function~\eqref{eq:tau},  we obtain
\begin{equation}
\label{eq:Bound S1}
 |S_1|  \le (d/ef)^{2n^2-(n+3)/2+o(1)}.
 \end{equation}
 
 Next, for $S_2$ we again use~\eqref{eq:tau} and Lemma~\ref{lem: SpL1}, to derive 
\begin{equation}
\label{eq:Bound S2}
  |S_2| =  \prod_{p\mid  e} \(p^{n^2} \left|S_{p}(b_{d,p}p^{-1} L)\right| \)
  \le  e^{2n^2-(n+1)/2+o(1)}.
 \end{equation}
 
 Finally,  for  $S_3$ we obviously have
 \begin{equation}
\label{eq:Bound S3}
  S_3 =  f^{2n^2}.
   \end{equation}
  Substituting the estimates~\eqref{eq:Bound S1}, \eqref{eq:Bound S2} 
  and~\eqref{eq:Bound S3} in~\eqref{eq:S S1S2S3}, we derive  the result. 
 \end{proof}

 We use Lemma~\ref{lem: Sd2 bound} in the following simplified form
 using the trivial bound  
$ ef^{(n+3)/2} \le (ef^2)^{(n+3)/4}$.

\begin{cor}
\label{cor: Sd2 bound}
For any square-free integer $d\ge 1$ and a linear form $ L(\vX)  \in \Z[\vX]$, 
we have
\[
|S_{d^2}(L)| \le d^{2n^2+o(1)}(D/d^2)^{(n+3)/4},
%%  d^{2n^2-(n+3)/2+o(1)}  D^{(n+3)/4}, 
\]
where   $D =  \gcd(L, d^2)$.
\end{cor}

 \section{Distribution of matrices}

\subsection{Discrepancy and the Koksma-Sz\"usz inequality}
\label{sec:discrepancy} 

We recall that the  {\it discrepancy\/}  $D(\Gamma)$ 
of  a sequence $\Gamma=(\bgamma_k)_{k= 1}^K$ of $K$ points 
in the half-open $s$-dimensional unit cube $[0,1)^s$, 
 is defined by
\[
 D(\Gamma) =\sup_{\cB \subseteq[0,1]^s}\left|\frac{A(\cB)}{K}- ( \beta_1-\alpha_1)   \cdots(\beta_s-\alpha_s)\right|,
\]
where the supremum is taken over all boxes  
\[
\cB = [\alpha_1, \beta_1] \times \ldots  \times [\alpha_s, \beta_s]\subseteq  [0,1)^s,
\]
and where  $A(\cB)$ is the number of elements of $\Gamma$, which fall in $\cB$, 
and 
\[ 
\vol \cB = ( \beta_1-\alpha_1)   \cdots(\beta_s-\alpha_s)
\] 
is the volume of $\cB$. 

One of the basic tools  %%used to study \u
in the uniformity of
distribution theory is the celebrated
{\it Koksma--Sz\"usz inequality\/}~\cite{Kok,Sz} 
(see also~\cite[Theorem~1.21]{DrTi}),
which links the discrepancy of a sequence of points to certain
exponential sums.

\begin{lemma}\label{lem:K-S}
For any integer $M \ge 1$, we have
\[
 D(\Gamma) \ll \frac{1}{M}
+\frac{1}{K}\sum_{0<\|\vm\|\le M}\frac{1}{r(\vm)}
\left| \sum_{k=1}^K \exp\(2\pi i \langle \vm, \bgamma_k\rangle\)\right|,
\]
where $ \langle \vm, \bgamma_k\rangle$ denotes the inner product, 
\[
\|\vm\|=  \max_{j=1, \ldots, s}|m_j|, \qquad
r(\vm) =  \prod_{j=1}^s \max\left\{|m_j|,1\right\},
\] 
and the sum is taken over all vectors
$\vm= (m_1,\ldots,m_s)\in\Z^s$ with $0<\|\vm\|\le M$, and the implied 
constant depends only on $s$. 
\end{lemma}

\subsection{Matrices with  determinants with a divisibility condition in boxes} 
\label{sec:boxes} 
We introduce the following extension of $N_n(m)$ from Section~\ref{sec:div det}.

For $H \ge 1$, we define 
\[
N_n(m; H) = \#\{A  \in \cM_n\(\Z;H\):~\det A \equiv 0 \bmod m\}. 
\]

We combine the bounds of exponential sums from Lemma~\ref{lem: Sd bound} and Corollary~\ref{cor: Sd2 bound}, respectively, 
with Lemma~\ref{lem:K-S} to derive asymptotic formulas for $N_n(m; H)$ in the cases when $m=d$ and $m = d^2$ 
for a square-free integer $d\ge 1$. 

\begin{lemma}
\label{lem: NdH}
For any square-free integer $d\ge 1$ and  any integer  $H \ge 1$ 
\[
N_n(d; H)  = N_n(d) \frac{(2H+1)^{n^2}} {d^{n^2}} 
+ O\(d^{n^2-2+1/n + o(1)}\).
\]
\end{lemma} 
\begin{proof} Let $m\ge 1$ be an integer and let $\vH = \(H_{ij}\)_{i,j=1}^n$ with some positive integers $H_{ij} < m/2$  (below we set $m=d$). 
We extend the definition of $N_n(m; H)$ to $N_n(m; \vH)$
which is the number of $n \times n$ matrices with
\[
A = \(a_{ij}\)_{i,j=1}^n \in \cM_n\(\Z\), \qquad  |a_{ij}| \le  H_{ij}  ,\  i,j =1, \dots, n, 
\]
and such that  $\det A \equiv 0 \bmod m$.  

Clearly, to establish the desired result, it is enough to show that 
 \begin{equation}
\label{eq:NH d}
N_n(d; \vH)  = N_n(d) \frac{1} {d^{n^2}} \prod_{i,j=1}^n\(2H_{ij}+1\) + O\(d^{n^2-2+1/n + o(1)}\).
\end{equation}
In fact we only need it when each $H_{ij}$ is either $d$ or $H$, but this does not 
simplify the argument.

To derive~\eqref{eq:NH}, we treat the problem of counting $N_n(d; \vH)$ as the discrepancy  
problem for the points $\(x_{ij}/d\)_{i,j=1}^n  \in (\R/\Z)^{n^2}$ with
\[ 
\det\(x_{ij}\)_{i,j=1}^n \equiv 0 \bmod {d}, \qquad x_{ij}\in \Z_{d}, \ i,j=1, \ldots, n, 
\]
embedded in the $n^2$-dimensional unit  torus $(\R/\Z)^{n^2}$.  In particular, our argument relies on 
Lemma~\ref{lem:K-S}, applied with $K = N_n(d)$ and with some $M \ge 1$, to be chosen later.

For a positive integer  $M \ge 1$, it is convenient to define  $\cL(M)$  as the set of all nonzero linear forms $L(\vX) \in \Z[\vX]$ 
of the form~\eqref{eq: LinForm} and with coefficients in the interval $[-M,M]$. 

Furthermore for $L$ as in~\eqref{eq: LinForm}, we denote
\[
\rho(L) =   \prod_{i,j=1}^n \max\left\{|a_{ij}|,1\right\}.
\]
 Hence, splitting each interval $[0,H_{ij})$ into intervals of length $d$, and applying Lemma~\ref{lem:K-S} to count points in incomplete boxes on the boundary, we have
\[
N_n(d; \vH)  = N_n(d) \frac{1} {d^{n^2}} \prod_{i,j=1}^n\(2H_{ij}+1\)+ O\(M^{-1} N_n(d)  + R\),
\]
where 
\[
R = \sum_{L \in  \cL(M)}  \frac{1}{\rho(L)} \abs{S_{d}(L)}.
\]
We note that Lemma~\ref{lem: Det = 0 Mod d^2} implies that 
 \begin{equation}
\label{eq:Nd}
N_n(d)  \le  d^{n^2} \prod_{p\mid d} \(p^{-1} + O\(p^{-2}\)\)  = d^{n^2-1 + o(1)}.
\end{equation}
Therefore, we have
 \begin{equation}
\label{eq:N and R d}
N_n(d; \vH)  =N_n(d) \frac{1} {d^{n^2}}  \prod_{i,j=1}^n\(2H_{ij}+1\)+ O\(d^{n^2-1 + o(1)}M^{-1}  + R\).
\end{equation}

Next for each positive divisor $D \mid d$  
we collect together the terms with $\gcd(L, d) = D$. 
Since $L \in \cL(M)$, we can assume that $D \le M$.
We also write 
 \begin{equation}
\label{eq: R Rm Rnm}
R = R_{\text{mon}} + R_{\text{non-mon}},
\end{equation}
where $R_{\text{mon}}$ and $R_{\text{non-mon}}$ are contributions to $R$ from  monomial and 
non-monomial linear forms $L$. 

Since $L \in \cL(M)$ is a nonzero linear form, the condition $\gcd(L, d) = D$ implies that 
each of its non-zero coefficients  is a non-zero multiple of $D$. 
It is now easy to see that 
\[
 \sum_{\substack{L \in \cL_{d}(M)\\
L~\text{monomial} \\\gcd(L, d) = D}}  \frac{1}{\rho(L)}  \ll  D^{-1} \log d  . 
 \]
By Lemma~\ref{lem: Sd bound}  and also using~\eqref{eq:tau},  this leads to the bound 
 \begin{equation}
\label{eq:  Rm}
R_{\text{mon}} = d^{n^2 -n+o(1)}  \sum_{\substack{D\mid d\\ D \le M}} D^{n} \sum_{\substack{L \in \cL_{d}(M)\\
L~\text{monomial} \\\gcd(L, d) = D}} \frac{1}{\rho(L)}  \le d^{n^2 -n+o(1)} M^{n-1} . 
 \end{equation}

%% Hence,  by~\eqref{eq:tau}, 
%% \begin{equation}
%%\label{eq: Rm}
%%R_{\text{mon}}  \le d^{n^2 -n+o(1)} M^{n-1} .
%% \end{equation}
%% 
 Furthermore, since non-monomial linear forms have at least two   non-zero coefficients, each of which  is a non-zero multiple of $D$,
 we see that 
 \[
 \sum_{\substack{L \in \cL_{d}(M)\\
L~\text{non-monomial} \\\gcd(L, d) = D}}  \frac{1}{\rho(L)}  \ll  D^{-2} (\log d)^{n^2} . 
 \]
 (the power of $\log d$ can easily  be reduced). Hence, by Lemma~\ref{lem: Sd bound} and~\eqref{eq:tau} again,
 \begin{equation}
\label{eq: Rnm}
\begin{split}
R_{\text{non-mon}} 
& = d^{n^2 -(n+1)/2+o(1)}  \sum_{\substack{D\mid d\\ D \le M}} D^{(n+1)/2} \sum_{\substack{L \in \cL_{d}(M)\\
L~\text{non-monomial} \\\gcd(L, d) = D}} \frac{1}{\rho(L)} \\
& \le d^{n^2 -(n+1)/2+o(1)} \(M^{(n-3)/2}+1\) . 
\end{split}
 \end{equation}
 
Substituting the bounds~\eqref{eq: Rm} and~\eqref{eq: Rnm} in~\eqref{eq: R Rm Rnm}
and the recalling~\eqref{eq:N and R d} we arrive to 
\[
N_n(d; \vH)  =N_n(d)  \frac{1} {d^{n^2}}  \prod_{i,j=1}^n\(2H_{ij}+1\)+ O\( d^{n^2+ o(1)} \sE \), 
\] 
where 
\[ 
\sE = d^{-1 } M^{-1}  + d^{ -n} M^{n-1}+
d^{ -(n+1)/2} M^{(n-3)/2}+ d^{ -(n+1)/2} .
\]
Comparing to potentially optimal values 
\[
M =d^{(n-1)/n} \mand M = d,
\]
we see that the first one leads to the stronger bound 
\begin{align*}
\sE& \ll  d^{-(2n-1)/n}+
d^{ -(5n-3)/(2n)}  + d^{ -(n+1)/2} \\
& \ll   d^{-(2n-1)/n}+d^{ -(n+1)/2} \ll  d^{-(2n-1)/n}. 
\end{align*}
Hence we now derive~\eqref{eq:NH d} and thus conclude the proof.  
\end{proof}

\begin{lemma}
\label{lem: Nd2H}
For any square-free integer $d\ge 1$ and  any integer  $H \ge 1$ 
\[
N_n(d^2; H)  = N_n(d^2) \frac{(2H+1)^{n^2}} {d^{2n^2}} 
+ O\(d^{2n^2-4(n+1)/(n+3) + o(1)}\).
\]
\end{lemma} 

\begin{proof} We follow the same path as in the proof of  Lemma~\ref{lem: NdH}. 
In particular, we continue to use the notations $\vH = \(H_{ij}\)_{i,j=1}^n$  and $N_n(m; \vH)$ (with $m=d^2$) and observe
that it is enough to show that 
 \begin{equation}
\label{eq:NH}
\begin{split}
N_n(d^2; \vH)  = N_n(d^2) \frac{1} {d^{2n^2}} \prod_{i,j=1}^n & \(2H_{ij}+1\)\\
& + O\(d^{2n^2-4(n+1)/(n+3) + o(1)}\)
\end{split}
\end{equation}
(this time we only need it when each $H_{ij}$ is either $d^2$ or $H$).

As in the proof of Lemma~\ref{lem: NdH},   with some $M \ge 1$, to be chosen later, by Lemma~\ref{lem:K-S} we have, 
\[
N_n(d^2; \vH)  = N_n(d^2) \frac{1} {d^{2n^2}} \prod_{i,j=1}^n\(2H_{ij}+1\)+ O\(M^{-1} N_n(d^2)  + R\),
\]
where 
\[
R = \sum_{L \in  \cL(M)}  \frac{1}{\rho(L)} \abs{S_{d^2}(L)}.
\]

This time, instead of~\eqref{eq:Nd},  Lemma~\ref{lem: Det = 0 Mod d^2} yields
 \begin{equation}
\label{eq:Nd2}
N_n(d^2)  \le  d^{2n^2} \prod_{p\mid d} \(p^{-2} + O\(p^{-3}\)\) \le  d^{2n^2-2 + o(1)}.
\end{equation}
Therefore, we have
 \begin{equation}
\label{eq:N and R}
N_n(d^2; \vH)  =N_n(d^2) \frac{1} {d^{2n^2}}  \prod_{i,j=1}^n\(2H_{ij}+1\)+ O\(M^{-1}  d^{2n^2-2 + o(1)} + R\).
\end{equation}
 
Next for each positive divisor $D \mid d^2$  
we collect together the terms with $\gcd(L, d^2) = D$. 
Since $L \in \cL(M)$, we can assume that $D \le M$.

Together with Corollary~\ref{cor: Sd2 bound} this leads to the bound 
\[
R = d^{2n^2 -(n+3)/2+o(1)}  \sum_{\substack{D\mid d^2\\ D \le M}} D^{(n+3)/4} \sum_{\substack{L \in \cL_{d^2}(M)\\\gcd(L, d^2) = D}} \frac{1}{\rho(L)} . 
\]
Since $L \in \cL(M)$ is a nonzero linear form, the condition $\gcd(L, d^2) = D$ implies that 
each of its non-zero coefficients  is a non-zero multiple of $D$. 
It is now easy to see that 
\[
 \sum_{\substack{L \in \cL(M)\\ \gcd(L, d^2) = D}} \frac{1}{\rho(L)}  \ll D^{-1} (\log d)^{n^2} .
 \]
Again, we note that one can certainly get a more precise bound here by classifying forms by the number $\nu \ge 1$
 of non-zero coefficients but this does not affect the final result. 
 
Hence, by~\eqref{eq:tau}, 
\[
R = d^{2n^2 -(n+3)/2+o(1)}    \sum_{\substack{D\mid d^2\\ D \le M}}  D^{(n-1)/4}   \le  M^{(n-1)/4} d^{2n^2 -(n+3)/2+o(1)} .
\]
Recalling~\eqref{eq:N and R}, we see that 
\begin{align*}
N_n(d^2; \vH)  &=  N_n(d^2) \frac{1} {d^{2n^2}}
 \prod_{i,j=1}^n\(2H_{ij}+1\)\\
&\qquad\qquad + O\(M^{-1}  d^{2n^2-2 + o(1)}
  +  M^{(n-1)/4} d^{2n^2 -(n+3)/2+o(1)} \).
\end{align*} 
Choosing $M = d^{2(n-1)/(n+3)}$, we derive~\eqref{eq:NH} and thus conclude the proof.  
\end{proof} 

 \section{Proof of Theorem~\ref{thm:SF-Det}}

 \subsection{Initial split}  Let $\mu(m)$ denote the {\it M{\"o}bius\/} function, that is, we have $\mu(m) = 0$ if 
 $m$ is not square-free, while otherwise $\mu(m) = (-1)^s$ where $s$ is the number of distinct 
 prime divisors of $m$.

 Then, by the inclusion-exclusion principle, we have 
 \[
S_n(H) =  \sum_{1 \le d \le \sqrt{n! H^n}} \mu(d) N_n(d^2; H) .
\]
In  the above sum we can certainly limit the range of  $d$ via the Hadamard inequality,  but this is 
not important for us. 

We now choose a parameter $\Delta\ge 1$ and split the above sums as 
 \begin{equation}
\label{eq:S ME}
S_n(H)  = \sM  + \sE, 
\end{equation}
where
\[
\sM =   \sum_{1\le d \le \Delta} \mu(d) N_n(d^2;  H) \mand 
\sE =  \sum_{ \Delta< d \le \sqrt{n! H^n}} \mu(d) N_n(d^2;  H) .
\] 

 \subsection{Main term} \label{sec:main term Sn}
 For the main term $\sM$ we use Lemma~\ref{lem: Nd2H} and write 
\begin{align*}
\sM & =   \sum_{1\le d \le \Delta}  \mu(d)  \(N_n(d^2) \frac{(2H+1)^{n^2}} {d^{2n^2}} 
+ O\(d^{2n^2-4(n+1)/(n+3) + o(1)}\)  \)\\
& = (2H+1)^{n^2}  \sum_{1\le d \le \Delta} N_n(d^2)   \frac{\mu(d) } {d^{2n^2}} 
+ O\(\Delta^{2n^2-(3n+1)/(n+3)+ o(1)}\)  .
\end{align*} 

Using the bound~\eqref{eq:Nd2}, and extending the summation in the above sum over all $d\ge 1$, we obtain 
\begin{align*}
\sM & =   (2H+1)^{n^2}  \sum_{d=1}^\infty  N_n(d^2)  \frac{\mu(d) } {d^{2n^2} }\\
&\qquad\qquad\qquad\qquad+ O\( H^{n^2}  \sum_{d >\Delta}  d^{-2 + o(1)}+  \Delta^{2n^2-(3n+1)/(n+3)+ o(1)}  \) \\
& =   (2H+1)^{n^2}  \sum_{d=1}^\infty  N_n(d^2)  \frac{\mu(d) } {d^{2n^2} } \\
&\qquad\qquad\qquad\qquad+  O\( H^{n^2} \Delta^{-1 + o(1)}+  \Delta^{2n^2-(3n+1)/(n+3)+ o(1)}  \) .
\end{align*}

 Using the muliplicativity of $\mu(d)$ and $N_n(d^2)/{d^{2n^2}}$, 
  and Lemma~\ref{lem: Det = 0 Mod d^2},  we derive 
\begin{align*}
  \sum_{d=1}^\infty  N_n(d^2)  \frac{\mu(d) } {d^{2n^2} } & =  \prod_{p} \(1 - N_n(p^2)/{p^{2n^2}}\) \\
   & = \prod_{p} \(1- \(1 - \prod_{j = 2}^{n+1} \(1-p^{-j}\)\)\)\\
 & =  \prod_{p}\,  \prod_{j = 2}^{n+1} \(1-p^{-j}\) =  \fS_n  . 
\end{align*}
Therefore,
 \begin{equation}
\label{eq:M-term}
\begin{split} 
\sM  = \fS_n & (2H+1)^{n^2}  \\
&+ O\( H^{n^2} \Delta^{-1 + o(1)}+  \Delta^{2n^2-(3n+1)/(n+3)+ o(1)}  \).
\end{split}
\end{equation}

 \subsection{Error term}
 \label{sec:err term SF}   To estimate the error term $\sE$, we write 
 \[
|\sE | \le   \sum_{ \Delta< d \le \sqrt{n! H^n}}   N_n(d^2; H),
\]
and notice that 
\[
N_n(d^2; H) \le \sum_{|a| \le n!H^n/d^2} \# \left\{A \in \cM_n\(\Z;H\):~\det A =ad^2\right\}.
\]
Hence, by Lemma~\ref{lem: Det} we have
\[
N_n(d^2; H) \le H^{n^2 - n + o(1)} \sum_{|a| \le n!H^n/d^2}1 \le H^{n^2 + o(1)}/d^2.
\]
We therefore derive 
 \begin{equation}
\label{eq:E-term}
|\sE | \le  H^{n^2 + o(1)}  \sum_{ \Delta< d \le \sqrt{n! H^n}} d^{-2} \le H^{n^2 + o(1)} \Delta^{-1}. 
\end{equation} 

 \subsection{Final optimisation} 
Substituting~\eqref{eq:M-term}  and~\eqref{eq:E-term} in~\eqref{eq:S ME}, we obtain 
\begin{align*}
S_n(H)  & = 2^{n^2} \fS_n H^{n^2}  \\
& \qquad \qquad + O\( H^{n^2} \Delta^{-1 + o(1)}+  \Delta^{2n^2-(3n+1)/(n+3)+ o(1)} + H^{n^2-1}  \),
\end{align*} 
which after choosing 
\[
\Delta = H^{n^2(n+3)/(2n^2(n+3)-2n+2)}
\]
implies the desired result.

 \section{Proof of Theorem~\ref{thm:Mob-Det}}

 \subsection{Initial split}  
 We proceed similarly to the proof of Theorem~\ref{thm:SF-Det}. We recall the formula 
 \[
 \varphi(m) = m \sum_{d \mid m} \frac{\mu(d)}{d}, 
 \]
 where, as before,   $\mu(d)$ denotes the {\it M{\"o}bius\/} function,  see,
 for example,~\cite[Equation~(16.1.3) and Section~16.3]{HardyWright}. Therefore, we obtain
 \[
\varPhi_n(H) =  \sum_{1 \le d \le \sqrt{n! H^n}} \frac{\mu(d)}{d} N_n(d; H) .
\]
In  the above sum we can certainly limit the range of  $d$ via the Hadamard inequality,  but this is 
not important for us. 

%%We now choose a parameter $\Delta$ and split the above sums as 
This time instead of~\eqref{eq:S ME}, for some  parameter $\Delta\ge 1$, we write
 \begin{equation}
\label{eq:F ME}
\varPhi_n(H)  = \sM  + \sE, 
\end{equation}
where
\[
\sM =   \sum_{1\le d \le \Delta} \frac{\mu(d)}{d} N_n(d; H)\mand 
\sE =  \sum_{ \Delta< d \le n! H^n} \frac{\mu(d)}{d} N_n(d; H).
\] 

 \subsection{Main term} 
 For the main term $\sM$ we use Lemma~\ref{lem: NdH} and write 
\begin{align*}
\sM & =   \sum_{1\le d \le \Delta}  \mu(d)  \(N_n(d) \frac{(2H+1)^{n^2}} {d^{n^2+1}} 
+ O\(d^{n^2-2 +1/n + o(1)}\)  \)\\
& = (2H+1)^{n^2}  \sum_{1\le d \le \Delta} N_n(d)   \frac{\mu(d) } {d^{n^2+1}} 
+ O\(\Delta^{n^2-1 +1/n  + o(1)}\)  .
\end{align*} 

Using the bound~\eqref{eq:Nd}, and extending the summation in the above sum over all $d\ge 1$, we obtain 
\begin{align*}
\sM & =  (2H+1)^{n^2}  \sum_{d=1}^\infty  N_n(d)   \frac{\mu(d) } {d^{n^2+1}}\\
&\qquad\qquad\qquad\qquad+ O\( H^{n^2}  \sum_{d >\Delta}  d^{-2 + o(1)}+  \Delta^{n^2-1 +1/n + o(1)}  \) \\
& =  (2H+1)^{n^2} \sum_{d=1}^\infty  N_n(d)  \frac{\mu(d) } {d^{n^2+1} } + O\( H^{n^2} \Delta^{-1 + o(1)}+  \Delta^{n^2-1 +1/n + o(1)}  \) .
\end{align*}

 Using the muliplicativity of $\mu(d)$ and $N_n(d)/{d^{n^2+1}}$, 
  and Lemma~\ref{lem: Det = 0 Mod d^2},  we derive 
\begin{align*}
  \sum_{d=1}^\infty  N_n(d)   \frac{\mu(d) } {d^{n^2+1}}& =  \prod_{p} \(1 - N_n(p)/{p^{n^2+1}}\) \\
   & = \prod_{p} \(1- \frac{1}{p}\(1 - \prod_{j = 1}^{n} \(1-p^{-j}\)\)\)\\
 & =  \prod_{p}\,  \(1- \frac{1}{p}\) \(1+  \frac{1}{p} \prod_{j = 2}^{n} \(1-p^{-j}\) \)=  \sigma_n. 
\end{align*}
Therefore,
 \begin{equation}
\label{eq:M-term Mob}
\sM  =  \sigma_n (2H+1)^{n^2} + O\( H^{n^2} \Delta^{-1 + o(1)}+  \Delta^{n^2-1 +1/n + o(1)}  \).
\end{equation}

 \subsection{Error term}  As in Section~\ref{sec:err term SF}, using that by Lemma~\ref{lem: Det}  we have
\[
N_n(d; H) \le H^{n^2 - n + o(1)} \sum_{|a| \le n!H^n/d}1 \le H^{n^2 + o(1)}/d, 
\]
we  derive 
 \begin{equation}
\label{eq:E-term Mob}
|\sE | \le  H^{n^2 + o(1)}  \sum_{ \Delta< d \le n! H^n} d^{-2} \le H^{n^2 + o(1)} \Delta^{-1}. 
\end{equation} 

 \subsection{Final optimisation} 
Substituting~\eqref{eq:M-term Mob}  and~\eqref{eq:E-term Mob} in~\eqref{eq:F ME}, we obtain 
\[
\varPhi_n(H)   = 2^{n^2} \sigma_n  H^{n^2}  + O\( H^{n^2} \Delta^{-1 + o(1)}+  \Delta^{n^2-1 +1/n + o(1)} + H^{n^2-1}  \),
\] 
which after choosing 
\[
\Delta = H^{n^3/(n^3+1)}
\]
implies the desired result. 

 \section{Proof of Theorem~\ref{thm:GCD-Det}}  
 Since our argument follows the proofs of Theorems~\ref{thm:SF-Det} 
 and~\ref{thm:Mob-Det} very closely, we only sketch it. 
 
 Namely, for a parameter $\Delta\ge 1$, instead of~\eqref{eq:S ME}, we now  have 
\[
R_n(H)  = \sM  + \sE, 
\] 
where
\[
\sM =   \sum_{1\le d \le \Delta} \mu(d) N_n(d;  H)^2 \mand 
\sE =  \sum_{ \Delta< d \le \sqrt{n! H^n}} \mu(d) N_n(d;  H)^2 .
\] 
We set 
\[
\Delta = H^{n^3/(n^3-n+1)}
\]
and observe that by~\eqref{eq:Nd},   the error term in Lemma~\ref{lem: NdH}  never exceeds the 
main terms (up to $d^{o(1)}$). Then, by Lemma~\ref{lem: NdH}, 
for $d \le \Delta$, we have 
\[
N_n(d; H)^2  = N_n(d)^2 \frac{(2H+1)^{2n^2}} {d^{2n^2}} 
+ O\(H^{n^2} d^{n^2-3+1/n + o(1)}\).
\]
Hence, as in Section~\ref{sec:main term Sn}, using the bound~\eqref{eq:Nd} again, we derive
\begin{align*}
\sM 
& =   (2H+1)^{2n^2}  \sum_{d=1}^\infty  N_n(d)^2  \frac{\mu(d) } {d^{2n^2} } \\
&\qquad\qquad\qquad\qquad+  O\( H^{2n^2} \Delta^{-1 + o(1)}+ H^{n^2} \Delta^{n^2-2+1/n+ o(1)}  \)\\
& =   \rho_n (2H+1)^{2n^2}   +  O\( H^{2n^2} \Delta^{-1 + o(1)}+ H^{n^2} \Delta^{n^2-2+1/n+ o(1)}  \)  .
\end{align*}  

For the error term, we have a full analogue of~\eqref{eq:E-term Mob}:
\[
|\cE|\le H^{2n^2} \Delta^{-1 + o(1)}.
\]
Recalling the choice of $\Delta$, we conclude the proof. 

\section*{Acknowledgement}

The authors would like to thank Alan Haynes for pointing out at a recurring oversight 
in the initial version of this paper  and also the referees for the very careful reading and 
helpful comments.

This work was supported, in part, by the Australian Research Council Grants  DP200100355 and DP230100530.  A.O. gratefully acknowledges the hospitality and support of the Max Planck Institute for Mathematics and Institut des Hautes {\'E}tudes Scientifiques, where her 
work has been carried out.

\end{document}